\documentclass[11pt,a4paper]{article}

\usepackage{inputenc}
\usepackage{amsmath}
\usepackage{bm}
\usepackage{bbold}
\usepackage{amsthm}
\usepackage{enumerate}

\usepackage{hyperref}

\title{Tropical Optimization Problems with Application to Project Scheduling with Minimum Makespan}

\author{Nikolai Krivulin\thanks{Faculty of Mathematics and Mechanics, St.~Petersburg State University, 28 Universitetsky Ave., St.~Petersburg, 198504, Russia, 
e-mail: nkk\textless at\textgreater math.spbu.ru.}
}

\date{}

\newtheorem{theorem}{Theorem}
\newtheorem{lemma}[theorem]{Lemma}

\begin{document}

\maketitle

\begin{abstract}
We consider multidimensional optimization problems in the framework of tropical mathematics. The problems are formulated to minimize a nonlinear objective function that is defined on vectors over an idempotent semifield and calculated by means of multiplicative conjugate transposition. We start with an unconstrained problem and offer two complete direct solutions to demonstrate different practicable argumentation schemes. The first solution consists of the derivation of a sharp lower bound for the objective function and the solving of an equation to find all vectors that yield the bound. The second is based on extremal properties of the spectral radius of matrices and involves the evaluation of this radius for a certain matrix. This solution is then extended to problems with boundary constraints that specify the feasible solution set by a double inequality, and with a linear inequality constraint given by a matrix. To illustrate one application of the results obtained, we solve problems in project scheduling under the minimum makespan criterion subject to various precedence constraints on the time of initiation and completion of activities in the project. Simple numerical examples are given to show the computational technique used for solutions. 
\\

\\

\end{abstract}

\section{Introduction}
\label{S-I}

Tropical (idempotent) mathematics, which focuses on the theory and applications of semirings with idempotent addition, offers a useful framework for the formulation and solution of real-world optimization problems in various fields of operations research, including project scheduling. Even the early works by Cuninghame-Green \cite{Cuninghamegreen1962Describing} and Giffler \cite{Giffler1963Scheduling} on tropical mathematics used optimization problems drawn from machine scheduling to motivate and illustrate the study.

In the last few decades, the theory and methods of tropical mathematics have received much attention, which resulted in many published works, such as recent monographs by Golan \cite{Golan2003Semirings}, Heidergott et al. \cite{Heidergott2006Maxplus}, Gondran and Minoux \cite{Gondran2008Graphs}, Butkovi\v{c} \cite{Butkovic2010Maxlinear}, McEneaney \cite{Mceneaney2010Maxplus}, and a great many contributed papers. Tropical optimization forms an important research domain within the field, which mainly concentrates on new solutions for problems in operations research. Applications in project scheduling remain of great concern in a number of researches, such as the works by Zimmermann \cite{Zimmermann2003Disjunctive,Zimmermann2006Interval}, Butkovi\v{c} et al. \cite{Butkovic2009Introduction,Butkovic2009Onsome,Aminu2012Nonlinear} and Krivulin \cite{Krivulin2014Amaximization,Krivulin2015Amultidimensional,Krivulin2013Explicit,Krivulin2015Extremal,Krivulin2014Aconstrained}. There are also applications in other areas, including those in location analysis developed by Cuninghame-Green \cite{Cuninghamegreen1991Minimax,Cuninghamegreen1994Minimax} and Krivulin \cite{Krivulin2011Anextremal,Krivulin2012Anew,Krivulin2013Direct,Krivulin2014Complete}, in decision making by Elsner and van den Driessche \cite{Elsner2004Maxalgebra,Elsner2010Maxalgebra}, Akian et al. \cite{Akian2012Tropical}, Gaubert et al. \cite{Gaubert2012Tropical} and Gursoy et al. \cite{Gursoy2013Theanalytic}, and in discrete event systems by Gaubert \cite{Gaubert1995Resource}, De Schutter \cite{Deschutter1996Maxalgebraic}, and De Schutter and van den Boom \cite{Deschutter2001Model}, to name only a few.

The purpose of this paper is twofold. First, we provide representative examples of optimization problems that are formulated and solved in the tropical mathematics setting to demonstrate general mathematical techniques used for solution. These simple, but not trivial, techniques can serve as a tool for addressing other similar problems in tropical optimization. Second, we show that the proposed methods of tropical optimization can well be applied to real-world problems in project scheduling to provide a unified formal description of the problems and direct closed-form solutions, which complement and supplement existing approaches.

In this paper, we consider multidimensional tropical optimization problems, which are formulated to minimize nonlinear objective functions defined on vectors over an idempotent semifield by means of a multiplicative conjugate transposition operator. We start with an unconstrained problem to propose two complete direct solutions to the problem, which offer different representations for the solution set. The first solution follows the approach developed in \cite{Krivulin2014Amaximization} to derive a sharp lower bound for the objective function and to solve an equation to find all vectors that yield the bound. The other one is based on extremal properties of the spectrum of matrices investigated in \cite{Krivulin2015Amultidimensional,Krivulin2014Aconstrained,Krivulin2015Extremal} and involves the evaluation of the spectral radius of a certain matrix. We show that, although these solutions are represented in different forms, they define the same solution set. The latter solution is then extended to solve the problem under constraints that specify lower and upper boundaries for the feasible solution set, and the problem under a linear inequality constraint given by a matrix.

To illustrate the application of the results obtained, we provide new exact solutions to problems in project scheduling under the minimum makespan objective. The problems are to minimize the overall duration of a project that consists of a number of activities to be performed in parallel subject to temporal precedence constraints, including start-finish, finish-start, early start and late finish (due date) constraints. The problems under consideration are known to have, in the usual setting, polynomial-time solutions in the form of computational algorithms (see, eg, overviews in Demeulemeester and Herroelen \cite{Demeulemeester2002Project}, T{\textquoteright}kindt and Billaut \cite{Tkindt2006Multicriteria} and Vanhoucke \cite{Vanhoucke2013Project}), and can be solved as linear programming problems as well. In contrast to these algorithmic solutions, the new ones are given directly in a compact vector form, which is ready for further analysis and straightforward computations.

The paper is organized as follows. Section~\ref{S-ME} describes example problems from project scheduling. In Section~\ref{S-DNR}, we offer a brief overview of key definitions and notation that underlie the development of solutions to the optimization problems and their applications in the subsequent sections. Section~\ref{S-PR} includes preliminary results, which provide a necessary prerequisite for the solution of the problems. In Section~\ref{S-UOP}, we first formulate an unconstrained optimization problem and solve the problem in two different ways. Furthermore, the solution is extended in Section~\ref{S-COP} to problems with constraints added. Section~\ref{AJS} presents application of the results to project scheduling. Illustrative numerical examples are given in Section~\ref{S-NE}.

\section{Motivational Examples}
\label{S-ME}

We start with real-world problems that are drawn from project scheduling under the minimum makespan criterion (see, e.g., \cite{Demeulemeester2002Project,Tkindt2006Multicriteria,Vanhoucke2013Project} for further details) as motivational and illustrative examples for the optimization problems under study.

Consider a project that involves $n$ activities operating under start-finish, finish-start, early start, and late finish (due date) temporal constraints. The start-finish constraints define the lower limit for the allowed time lag between the initiation of one activity and the completion of another. The activities are assumed to be completed as early as possible within the start-finish constraints. The finish-start constraints determine the minimum time lag between the completion of one activity and the initiation of another. The early start and late finish constraints specify, respectively, the earliest possible initiation time and the latest possible completion time for every activity.

Below, we first examine a problem that has only start-finish constraints, and then extend the result obtained to problems with additional constraints. 

For each activity $i=1,\ldots,n$, we denote the initiation time by $x_{i}$ and the completion time by $y_{i}$. Let $c_{ij}$ be the minimum time lag between the initiation of activity $j=1,\ldots,n$ and the completion of $i$. If $c_{ij}$ is not given for some $j$, we put $c_{ij}=-\infty$. The completion time of activity $i$ must satisfy the start-finish relations
$$
x_{j}+c_{ij}
\leq
y_{i},
\qquad
j=1,\ldots,n,
$$ 
where at least one inequality holds as equality. Combining the relations gives
$$
y_{i}
=
\max_{1\leq j\leq n}(x_{j}+c_{ij}).
$$

The makespan is defined as the duration between the earliest initiation time and the latest completion time in the project, and takes the form
$$
\max_{1\leq i\leq n}y_{i}
-
\min_{1\leq i\leq n}x_{i}
=
\max_{1\leq i\leq n}y_{i}
+
\max_{1\leq i\leq n}(-x_{i}).
$$

After substitution of $y_{i}$, the problem of scheduling under the start-finish constraints and the minimum makespan criterion can be formulated as follows: given $c_{ij}$ for $i,j=1,\ldots,n$, find $x_{1},\ldots,x_{n}$ that
\begin{equation}
\begin{aligned}
&
\text{minimize}
&&
\max_{1\leq i\leq n}\max_{1\leq j\leq n}(x_{j}+c_{ij})
+
\max_{1\leq i\leq n}(-x_{i}).
\end{aligned}
\label{P-minmaxxcmaxx}
\end{equation}

Furthermore, we consider the problem with the early start and late finish constraints added. For each activity $i=1,\ldots,n$, let $g_{i}$ be the earliest possible time to start, and $f_{i}$ the latest possible time to finish (the due date) for activity $i$. The early start and late finish constraints imply the inequalities
$$
g_{i}
\leq
x_{i},
\qquad
y_{i}
=
\max_{1\leq j\leq n}(x_{j}+c_{ij})
\leq
f_{i},
$$
which, combined with the objective function, yield a problem that is given by
\begin{equation}
\begin{aligned}
&
\text{minimize}
&&
\max_{1\leq i\leq n}\max_{1\leq j\leq n}(x_{j}+c_{ij})
+
\max_{1\leq i\leq n}(-x_{i}),
\\
&
\text{subject to}
&&
\max_{1\leq j\leq n}(x_{j}+c_{ij})
\leq
f_{i},
\\
&&&
g_{i}
\leq
x_{i},
\qquad
i=1,\ldots,n.
\end{aligned}
\label{P-minmaxxcmaxx-xcf-gx}
\end{equation}

Finally, suppose that, in the project under consideration, the late finish constraints are replaced by finish-start constraints. For each activity $i=1,\ldots,n$, we denote by $d_{ij}$ the minimum allowed time lag between the completion of activity $j$ and initiation of $i$. We take $d_{ij}=-\infty$ if the time lag is not specified.

The finish-start constraints are given by the inequalities
$$
y_{j}+d_{ij}
\leq
x_{i},
\qquad
j=1,\ldots,n.
$$ 

Furthermore, we substitute $y_{j}$ from the start-finish constraints and combine the inequalities into one to write 
$$
\max_{1\leq j\leq n}(\max_{1\leq k\leq n}(x_{k}+c_{jk})+d_{ij})
\leq
x_{i}.
$$

The scheduling problem under finish-start and early start constraints can now be formulated as
\begin{equation}
\begin{aligned}
&
\text{minimize}
&&
\max_{1\leq i\leq n}\max_{1\leq j\leq n}(x_{j}+c_{ij})
+
\max_{1\leq i\leq n}(-x_{i}),
\\
&
\text{subject to}
&&
\max_{1\leq j\leq n}(\max_{1\leq k\leq n}(x_{k}+c_{jk})+d_{ij})
\leq
x_{i},
\\
&&&
g_{i}
\leq
x_{i},
\qquad
i=1,\ldots,n.
\end{aligned}
\label{P-minmaxxcmaxx-xcdx-gx}
\end{equation}

To conclude this section, we note that it is not difficult to rearrange the above optimization problems as linear programming problems. Although this approach offers a simple way to obtain solutions of these problems by an appropriate linear programming algorithm, it does not guarantee the solutions to be obtained in a direct closed form. Below, we propose another technique that is based on the formulation and solution of the problems in terms of tropical mathematics as tropical optimization problems. In contrast to the existing algorithmic solutions, this technique provides complete, direct solutions in a compact vector form that offers a solid basis for further analysis and practical implementation of the results.

\section{Definitions, Notation and General Remarks}
\label{S-DNR}

We first give a brief overview of main definitions and notation of tropical mathematics to provide a formal framework for the solution of the optimization problems in the next sections. The overview is mainly based on the results in \cite{Krivulin2009Methods}. For additional details, insights and references, one can consult \cite{Golan2003Semirings,Heidergott2006Maxplus,Akian2007Maxplus,Gondran2008Graphs,Speyer2009Tropical,Butkovic2010Maxlinear,Mceneaney2010Maxplus}.

Let $\mathbb{X}$ be a set endowed with two associative and commutative operations, $\oplus$ (addition) and $\otimes$ (multiplication), and equipped with additive and multiplicative neutral elements, $\mathbb{0}$ (zero) and $\mathbb{1}$ (one). Addition is idempotent, which yields $x\oplus x=x$ for every $x\in\mathbb{X}$. Multiplication distributes over addition and is invertible to provide each nonzero $x\in\mathbb{X}$ with its inverse $x^{-1}$ such that $x\otimes x^{-1}=\mathbb{1}$. The system $(\mathbb{X},\oplus,\otimes,\mathbb{0},\mathbb{1})$ is commonly referred to as the idempotent semifield, and retains certain properties of the usual fields.

We assume that the semifield is linearly ordered by an order that is consistent with the partial order induced by idempotent addition to define $x\leq y$ if and only if $x\oplus y=y$. From here on, we use the relation symbols as well as $\max$ and $\min$ operators in the sense of this definition. Specifically, it follows from the definition that $x\oplus y=\max(x,y)$. Moreover, in terms of the above partial order, both operations $\oplus$ and $\otimes$ are monotone in each argument.

As usual, the integer power specifies iterated multiplication, and is defined by $x^{p}=x^{p-1}\otimes x$, $x^{-p}=(x^{-1})^{p}$, $x^{0}=\mathbb{1}$ and $\mathbb{0}^{p}=\mathbb{0}$ for each nonzero $x$ and integer $p\geq1$. Moreover, the semifield is taken to be algebraically complete (radicable), which yields the existence of a solution of the equation $x^{m}=a$ for all $a\ne\mathbb{0}$ and integer $m$, and hence the existence of the root $a^{1/m}$. In the rest of the paper, we omit the multiplication sign $\otimes$ for the sake of simplicity.

Characteristic examples of the idempotent semifield under consideration include $\mathbb{R}_{\max,+}=(\mathbb{R}\cup\{-\infty\},\max,+,-\infty,0)$, $\mathbb{R}_{\min,+}=(\mathbb{R}\cup\{+\infty\},\min,+,+\infty,0)$, $\mathbb{R}_{\max,\times}=(\mathbb{R}_{+}\cup\{0\},\max,\times,0,1)$ and $\mathbb{R}_{\min,\times}=(\mathbb{R}_{+}\cup\{+\infty\},\min,\times,+\infty,1)$, where $\mathbb{R}$ is the set of real numbers and $\mathbb{R}_{+}=\{x\in\mathbb{R}|x>0\}$.

Specifically, in the semifield $\mathbb{R}_{\max,+}$, we have $\oplus=\max$, $\otimes=+$, $\mathbb{0}=-\infty$, and $\mathbb{1}=0$. Each nonzero $x\in\mathbb{R}_{\max,+}$ has the inverse $x^{-1}$, which is equal to $-x$ in the usual notation. The power $x^{y}$ is defined for all nonzero $x$ and $y$, and corresponds to the arithmetic product $xy$. Specifically, if $m\ne0$, then the expression $x^{1/m}$ means the usual division $x/m$. The partial order associated with the idempotent addition coincides with the usual linear order on $\mathbb{R}$.

Let $\mathbb{X}^{m\times n}$ be the set of matrices with $m$ rows and $n$ columns over $\mathbb{X}$. A matrix with all entries equal to $\mathbb{0}$ is the zero matrix. If a matrix has no zero rows (columns), then it is called row-regular (column-regular).

Addition and multiplication of conforming matrices and scalar multiplication are defined by the standard rules with the scalar operations $\oplus$ and $\otimes$ used in place of the ordinary addition and multiplication. These matrix operations are monotone with respect to the order relations defined component-wise. As in the scalar case, we drop the sign $\otimes$ when representing the operations of matrix multiplication.

For any matrix $\bm{A}\in\mathbb{X}^{m\times n}$, the transpose of $\bm{A}$ is the matrix $\bm{A}^{T}\in\mathbb{X}^{n\times m}$.

Consider the set $\mathbb{X}^{n\times n}$ of square matrices of order $n$. A matrix with $\mathbb{1}$ on the diagonal and $\mathbb{0}$ elsewhere is the identity matrix denoted $\bm{I}$.

The integer powers of any matrix $\bm{A}\in\mathbb{X}^{n\times n}$ are defined as $\bm{A}^{0}=\bm{I}$ and $\bm{A}^{p}=\bm{A}^{p-1}\bm{A}$ for any integer $p\geq1$.

Tropical analogues of the trace and the norm of a matrix $\bm{A}=(a_{ij})$ are respectively given by
$$
\mathop\mathrm{tr}\bm{A}=\bigoplus_{i=1}^{n}a_{ii},
\qquad
\|\bm{A}\|
=
\bigoplus_{i=1}^{n}\bigoplus_{j=1}^{n}a_{ij}.
$$

For any matrices $\bm{A}$ and $\bm{B}$ and a scalar $x$, we obviously have
$$
\mathop\mathrm{tr}(\bm{A}\bm{B})
=
\mathop\mathrm{tr}(\bm{B}\bm{A}),
\qquad
\mathop\mathrm{tr}(x\bm{A})
=
x\mathop\mathrm{tr}\bm{A},
\qquad
\|\bm{A}\oplus\bm{B}\|
=
\|\bm{A}\|\oplus\|\bm{B}\|.
$$

Furthermore, we define a function that assigns to each matrix $\bm{A}$ a scalar
\begin{equation}
\mathop\mathrm{Tr}(\bm{A})
=
\bigoplus_{m=1}^{n}\mathop\mathrm{tr}\bm{A}^{m}.
\label{E-TrA}
\end{equation}

Provided the condition $\mathop\mathrm{Tr}(\bm{A})\leq\mathbb{1}$ holds, the asterate of $\bm{A}$ (also known as the Kleene star) is the matrix given by
\begin{equation}
\bm{A}^{\ast}
=
\bigoplus_{m=0}^{n-1}\bm{A}^{m}.
\label{E-Aast}
\end{equation}

Under the above condition, the asterate possesses a useful property that takes the form of the inequality (the Carr\`{e} inequality)
$$
\bm{A}^{m}
\leq
\bm{A}^{\ast},
\qquad
m\geq0.
$$

The set of column vectors of size $n$ over $\mathbb{X}$ is denoted by $\mathbb{X}^{n}$. The vectors are considered below as column vectors unless indicated otherwise.

A vector with all entries equal to $\mathbb{0}$ is the zero vector denoted $\bm{0}$. A vector is regular if it has no zero elements. It is clear that if $\bm{x}\in\mathbb{X}^{n}$ is a regular vector and $\bm{A}\in\mathbb{X}^{n\times n}$ is a column-regular matrix, then the row vector $\bm{x}^{T}\bm{A}$ is regular.

For any nonzero column vector $\bm{x}=(x_{i})\in\mathbb{X}^{n}$, the multiplicative conjugate transpose is the row vector $\bm{x}^{-}=(x_{i}^{-})$ with elements $x_{i}^{-}=x_{i}^{-1}$ if $x_{i}\ne\mathbb{0}$, and $x_{i}^{-}=\mathbb{0}$ otherwise. The conjugate transposition possesses certain useful properties, which are not difficult to verify. Specifically, for any nonzero vector $\bm{x}$, we have the obvious equality $\bm{x}^{-}\bm{x}=\mathbb{1}$. For any regular vectors $\bm{x}$ and $\bm{y}$ of the same order, the component-wise inequality $\bm{x}\bm{y}^{-}\geq(\bm{x}^{-}\bm{y})^{-1}\bm{I}$ is valid as well.

A scalar $\lambda\in\mathbb{X}$ is an eigenvalue of the matrix $\bm{A}\in\mathbb{X}^{n\times n}$ if there exists a nonzero vector $\bm{x}\in\mathbb{X}^{n}$ such that $\bm{A}\bm{x}=\lambda\bm{x}$. The maximum eigenvalue is called the spectral radius of $\bm{A}$ and calculated by the formula
\begin{equation}
\lambda
=
\bigoplus_{m=1}^{n}\mathop\mathrm{tr}\nolimits^{1/m}(\bm{A}^{m}).
\label{E-lambda}
\end{equation}

Finally, note that, for any matrix $\bm{A}$, we can write $\|\bm{A}\|=\bm{1}^{T}\bm{A}\bm{1}$, where $\bm{1}=(\mathbb{1},\ldots,\mathbb{1})^{T}$. If $\bm{A}=\bm{x}\bm{y}^{T}$, where $\bm{x}$ and $\bm{y}$ are column vectors, then $\|\bm{A}\|=\|\bm{x}\|\|\bm{y}\|$.

\section{Preliminary Results}
\label{S-PR}

We now present preliminary results concerning the solution of algebraic and optimization problems in the tropical mathematics setting to be used below.

First, we assume that, given a vector $\bm{a}\in\mathbb{X}^{n}$ and a scalar $d\in\mathbb{X}$, we need to obtain vectors $\bm{x}\in\mathbb{X}^{n}$ to satisfy the equation
\begin{equation}
\bm{a}^{T}\bm{x}
=
d.
\label{E-axd}
\end{equation}

A complete solution to the problem can be described as follows \cite{Krivulin2009Methods}.

\begin{lemma}
\label{L-axd}
Let $\bm{a}=(a_{i})$ be a regular vector and $d\ne\mathbb{0}$ be a scalar. Then, the solution of equation \eqref{E-axd} forms a family of solutions each defined for one of $k=1,\ldots,n$ as a set of vectors $\bm{x}=(x_{i})$ with components
$$
x_{k}
=
a_{k}^{-1}d,
\qquad
x_{i}
\leq
a_{i}^{-1}d,
\quad
i\ne k.
$$
\end{lemma}

Given a matrix $\bm{A}\in\mathbb{X}^{m\times n}$ and a vector $\bm{d}\in\mathbb{X}^{m}$, consider the problem to find regular vectors $\bm{x}\in\mathbb{X}^{n}$ that satisfy the inequality
\begin{equation}
\bm{A}\bm{x}
\leq
\bm{d}.
\label{I-Axd}
\end{equation}

The next statement offers a solution that is obtained as a consequence of the solution to the corresponding equation \cite{Krivulin2009Methods}, and by independent proof \cite{Krivulin2013Direct}. 

\begin{lemma}
\label{L-Axd}
Let $\bm{A}$ be a column-regular matrix and $\bm{d}$ a regular vector. Then, all regular solutions to inequality \eqref{I-Axd} are given by
\begin{equation*}
\bm{x}
\leq
(\bm{d}^{-}\bm{A})^{-}.
\label{I-xdA}
\end{equation*}
\end{lemma}

Furthermore, assume that, for a given matrix $\bm{A}\in\mathbb{X}^{n\times n}$, we need to find regular solutions $\bm{x}\in\mathbb{X}^{n}$ to the problem
\begin{equation}
\begin{aligned}
&
\text{minimize}
&&
\bm{x}^{-}\bm{A}\bm{x}.
\end{aligned}
\label{P-minxAx}
\end{equation}

A complete solution to \eqref{P-minxAx} is provided by the following result \cite{Krivulin2015Amultidimensional,Krivulin2014Aconstrained,Krivulin2015Extremal}.

\begin{lemma}
\label{L-minxAx}
Let $\bm{A}$ be a matrix with spectral radius $\lambda>\mathbb{0}$. Then, the minimum value in problem \eqref{P-minxAx} is equal to $\lambda$ and all regular solutions are given by
$$
\bm{x}
=
(\lambda^{-1}\bm{A})^{\ast}\bm{u},
\qquad
\bm{u}
>
\bm{0}.
$$
\end{lemma}

We conclude with solutions obtained in \cite{Krivulin2014Aconstrained} for constrained versions of problem \eqref{P-minxAx}. First, we offer a solution to the problem: given a matrix $\bm{A}\in\mathbb{X}^{n\times n}$ and vectors $\bm{p},\bm{q}\in\mathbb{X}^{n}$, find regular vectors $\bm{x}\in\mathbb{X}^{n}$ that
\begin{equation}
\begin{aligned}
&
\text{minimize}
&&
\bm{x}^{-}\bm{A}\bm{x},
\\
&
\text{subject to}
&&
\bm{g}
\leq
\bm{x}
\leq
\bm{h}.
\end{aligned}
\label{P-minxAxgxh}
\end{equation}

\begin{theorem}
\label{T-minxAxgxh}
Let $\bm{A}$ be a matrix with spectral radius $\lambda>\mathbb{0}$, and $\bm{h}$ be a regular vector such that $\bm{h}^{-}\bm{g}\leq\mathbb{1}$. Then, the minimum in problem \eqref{P-minxAxgxh} is equal to 
$$
\theta
=
\lambda
\oplus
\bigoplus_{m=1}^{n}(\bm{h}^{-}\bm{A}^{m}\bm{g})^{1/m},
$$
and all regular solutions of the problem are given by
$$
\bm{x}
=
(\theta^{-1}\bm{A})^{\ast}\bm{u},
\qquad
\bm{g}
\leq
\bm{u}
\leq
(\bm{h}^{-}(\theta^{-1}\bm{A})^{\ast})^{-}.
$$
\end{theorem}

Finally, we present a solution to the following problem. Given matrices $\bm{A},\bm{B}\in\mathbb{X}^{n\times n}$ and a vector $\bm{g}\in\mathbb{X}^{n}$, we look for regular vectors $\bm{x}\in\mathbb{X}^{n}$ to
\begin{equation}
\begin{aligned}
&
\text{minimize}
&&
\bm{x}^{-}\bm{A}\bm{x},
\\
&
\text{subject to}
&&
\bm{B}\bm{x}\oplus\bm{g}
\leq
\bm{x}.
\end{aligned}
\label{P-minxAxBxgx}
\end{equation}

\begin{theorem}
\label{T-minxAxBxgx}
Let $\bm{A}$ be a matrix with spectral radius $\lambda>\mathbb{0}$, and $\bm{B}$ be a matrix such that $\mathop\mathrm{Tr}(\bm{B})\leq\mathbb{1}$. Then, the minimum value in problem \eqref{P-minxAxBxgx} is equal to
$$
\theta
=
\bigoplus_{k=1}^{n}\mathop{\bigoplus\hspace{1.1em}}_{0\leq i_{1}+\cdots+i_{k}\leq n-k}\mathop\mathrm{tr}\nolimits^{1/k}(\bm{A}\bm{B}^{i_{1}}\cdots\bm{A}\bm{B}^{i_{k}}),
$$
and all regular solutions of the problem are given by
$$
\bm{x}
=
(\theta^{-1}\bm{A}\oplus\bm{B})^{\ast}\bm{u},
\qquad
\bm{u}
\geq
\bm{g}.
$$
\end{theorem}

\section{Unconstrained Optimization Problem}
\label{S-UOP}

In this section, we examine an unconstrained multidimensional optimization problem formulated in the tropical mathematics setting as follows. Given vectors $\bm{p},\bm{q}\in\mathbb{X}^{n}$, the  problem is to find regular vectors $\bm{x}\in\mathbb{X}^{n}$ that 
\begin{equation}
\begin{aligned}
&
\text{minimize}
&&
\bm{q}^{-}\bm{x}\bm{x}^{-}\bm{p}.
\end{aligned}
\label{P-minqxxp}
\end{equation}

Suppose $p_{i}$, $q_{i}$ and $x_{i}$ represent the elements of the vectors $\bm{p}$, $\bm{q}$ and $\bm{x}$, respectively. Then, the objective function in problem \eqref{P-minqxxp} can be written as
$$
\bm{q}^{-}\bm{x}\bm{x}^{-}\bm{p}
=
(q_{1}^{-1}x_{1}\oplus\cdots\oplus q_{n}^{-1}x_{n})(x_{1}^{-1}p_{1}\oplus\cdots\oplus x_{n}^{-1}p_{n}).
$$

Below, we offer two direct complete solutions to the problem under fairly general assumptions. We show that, although these solutions have different forms, both forms determine the same solution set.

\subsection{Straightforward Solution}

We start with a solution based on the derivation of a lower bound for the objective function, and on the solution of an equation that puts the function equal to the bound to find all solution vectors.

\begin{theorem}\label{T-minqxxp}
Let $\bm{p}$ be a nonzero vector and $\bm{q}$ a regular vector. Then, the minimum value in problem \eqref{P-minqxxp} is equal to 
\begin{equation*}
\Delta
=
\bm{q}^{-}\bm{p},
\end{equation*}
and all regular solutions of the problem are given by
\begin{equation}
\alpha\bm{p}
\leq
\bm{x}
\leq
\alpha\Delta\bm{q},
\qquad
\alpha>\mathbb{0}.
\label{I-alphapxalphaDeltaq}
\end{equation}
\end{theorem}
\begin{proof}
First, we find the minimum of the objective function in the problem by using the properties of the conjugate transposition. With the inequality $\bm{x}\bm{x}^{-}\geq\bm{I}$, which is valid for any regular vector $\bm{x}$, we derive a lower bound
$$
\bm{q}^{-}\bm{x}\bm{x}^{-}\bm{p}
\geq
\bm{q}^{-}\bm{p}
=
\Delta.
$$

Note that, since $\bm{p}$ is nonzero and $\bm{q}$ is regular, we have $\Delta>\mathbb{0}$.

It remains to verify that this bound is attained at a vector $\bm{x}$, say $\bm{x}=\Delta\bm{q}$. Indeed, substitution into the objective function and the equality $\bm{q}^{-}\bm{q}=\mathbb{1}$ yield
$$
\bm{q}^{-}\bm{x}\bm{x}^{-}\bm{p}
=
\Delta(\bm{q}^{-}\bm{q})\Delta^{-1}(\bm{q}^{-}\bm{p})
=
\bm{q}^{-}\bm{p}
=
\Delta.
$$

To obtain all regular vectors $\bm{x}$ that solve the problem, we examine the equation
$$
\bm{q}^{-}\bm{x}\bm{x}^{-}\bm{p}
=
\Delta.
$$

It is clear that, if $\bm{x}$ is a solution, then so is $\alpha\bm{x}$ for any $\alpha>\mathbb{0}$, and hence all solutions of the equation are scale-invariant. 

Furthermore, we take an arbitrary $\alpha>\mathbb{0}$ and rewrite the equation in an equivalent form as the system of two equations
$$
\bm{q}^{-}\bm{x}
=
\alpha\Delta,
\qquad
\bm{x}^{-}\bm{p}
=
\alpha^{-1}.
$$

Taking into account that all solutions are scale-invariant, we put $\alpha=\mathbb{1}$ to further reduce the system as
\begin{equation}
\bm{q}^{-}\bm{x}
=
\Delta,
\qquad
\bm{x}^{-}\bm{p}
=
\mathbb{1}.
\label{E-qxDealta-xp1}
\end{equation}
 
According to Lemma~\ref{L-axd}, the solutions to the first equation form a family of solutions $\bm{x}=(x_{i})$, each defined for one of $k=1,\ldots,n$ by the conditions
$$
x_{k}
=
\Delta q_{k},
\qquad
x_{i}
\leq
\Delta q_{i},
\quad
i\ne k.
$$

We now find those solutions from the family which satisfy the second equation at \eqref{E-qxDealta-xp1}. Note that $\bm{x}=\Delta\bm{q}$ solves the problem and thus satisfies this equation.

Consider the minimum value of the problem and write
$$
\Delta
=
\bm{q}^{-}\bm{p}
=
\bigoplus_{i=1}^{n}q_{i}^{-1}p_{i}
=
q_{k}^{-1}p_{k},
$$
where $k$ is an index that yields the maximum of $q_{i}^{-1}p_{i}$ over all $i=1,\ldots,n$.

We denote by $K$ the set of all such indices that produce $\Delta$, and note that $p_{k}\ne\mathbb{0}$ for all $k\in K$.

We now verify that all solutions to the second equation must have $x_{k}=\Delta q_{k}$ if $k\in K$. 
Assuming the contrary, let $k$ be an index in $K$ to satisfy the condition
$$
x_{k}
<
\Delta q_{k}
=
(\bm{q}^{-}\bm{p})q_{k}
=
q_{k}^{-1}p_{k}q_{k}
=
p_{k}.
$$

Then, for the left hand side of the second equation at \eqref{E-qxDealta-xp1}, we have
$$
\bm{x}^{-}\bm{p}
=
x_{1}^{-1}p_{1}\oplus\cdots\oplus x_{n}^{-1}p_{n}
\geq
x_{k}^{-1}p_{k}
>
p_{k}^{-1}p_{k}
=
\mathbb{1},
$$
and thus the equation is not valid anymore and becomes a strict inequality.

Furthermore, for all $i\not\in K$ if any, we can take $x_{i}\leq\Delta q_{i}$ but not too small to keep the condition $\bm{x}^{-}\bm{p}\leq\mathbb{1}$. It follows from this condition that
$$
\mathbb{1}
\geq
\bm{x}^{-}\bm{p}
=
x_{1}^{-1}p_{1}\oplus\cdots\oplus x_{n}^{-1}p_{n}
\geq
x_{i}^{-1}p_{i}.
$$

To satisfy the condition when $p_{i}\ne\mathbb{0}$, we have to take $x_{i}\geq p_{i}$. With $p_{i}=\mathbb{0}$, the term $\bm{x}^{-}\bm{p}$ does not depend on $x_{i}$, and thus any $x_{i}\geq\mathbb{0}=p_{i}$ meets the condition.

We can summarize the above consideration as follows. All solutions to the problem are vectors $\bm{x}=(x_{i})$ that satisfy the conditions
\begin{align*}
x_{i}
&=
\Delta q_{i},
\qquad
i\in K
\\
p_{i}
\leq
x_{i}
&\leq
\Delta q_{i},
\qquad
i\not\in K.
\end{align*}

Since we have $x_{i}=\Delta q_{i}=q_{i}^{-1}p_{i}q_{i}=p_{i}$ for all $i\in K$, the solution can be written as one double inequality $p_{i}\leq x_{i}\leq\Delta q_{i}$ for all $i=1,\ldots,n$, or, in vector form, as the inequality
$$
\bm{p}
\leq
\bm{x}
\leq
\Delta\bm{q}.
$$

Considering that each solution is scale-invariant, we arrive at \eqref{I-alphapxalphaDeltaq}.
\end{proof}

\subsection{Solution Using Spectral Radius}

To provide another solution to problem \eqref{P-minqxxp}, we first put the objective function in the equivalent form
$$
\bm{q}^{-}\bm{x}\bm{x}^{-}\bm{p}
=
\bm{x}^{-}\bm{p}\bm{q}^{-}\bm{x},
$$
and then rewrite the problem as
\begin{equation}
\begin{aligned}
&
\text{minimize}
&&
\bm{x}^{-}\bm{p}\bm{q}^{-}\bm{x}.
\end{aligned}
\label{P-minxpqx}
\end{equation}

The problem now becomes a special case of problem \eqref{P-minxAx} with $\bm{A}=\bm{p}\bm{q}^{-}$, and can therefore be solved using Lemma~\ref{L-minxAx}.

\begin{theorem}\label{T-minxpqx}
Let $\bm{p}$ be a nonzero vector and $\bm{q}$ a regular vector. Then, the minimum value in problem \eqref{P-minqxxp} is equal to
$$
\Delta
=
\bm{q}^{-}\bm{p},
$$
and all regular solutions of the problem are given by
\begin{equation}
\bm{x}
=
(\bm{I}
\oplus
\Delta^{-1}\bm{p}\bm{q}^{-})
\bm{u},
\qquad
\bm{u}
>
\bm{0}.
\label{E-xIqppq}
\end{equation}
\end{theorem}
\begin{proof}
We examine the problem in the form of \eqref{P-minxpqx}. To apply Lemma~\ref{L-minxAx}, we take the matrix $\bm{A}=\bm{p}\bm{q}^{-}$ and calculate
$$
\bm{A}^{m}
=
(\bm{q}^{-}\bm{p})^{m-1}\bm{p}\bm{q}^{-},
\qquad
\mathop\mathrm{tr}\bm{A}^{m}
=
(\bm{q}^{-}\bm{p})^{m},
\qquad
m=1,\ldots,n.
$$

Let $\Delta$ be the spectral radius of the matrix $\bm{A}$. Using formula \eqref{E-lambda}, we obtain $\Delta=\bm{q}^{-}\bm{p}$, which, due to Lemma~\ref{L-minxAx}, presents the minimum value in the problem.

To describe the solution set, we calculate $(\Delta^{-1}\bm{A})^{m}=\Delta^{-m}\bm{A}^{m}=\Delta^{-1}\bm{p}\bm{q}^{-}$, and then employ \eqref{E-Aast} to derive the matrix
$$
(\Delta^{-1}\bm{A})^{\ast}
=
\bm{I}\oplus\Delta^{-1}\bm{p}\bm{q}^{-}.
$$

Finally, the application of Lemma~\ref{L-minxAx} gives solution \eqref{E-xIqppq}.
\end{proof}

Note that, although the solution sets offered by Theorems~\ref{T-minqxxp} and \ref{T-minxpqx} look different, it is not difficult to verify that they are the same.

Let us take any vector $\bm{u}$ and ascertain that $\bm{x}$, which is given by \eqref{E-xIqppq}, satisfies inequality \eqref{I-alphapxalphaDeltaq} for some $\alpha$. Indeed, if we put $\alpha=\Delta^{-1}(\bm{q}^{-}\bm{u})$, then we have 
$$
\bm{x}
=
(\bm{I}\oplus\Delta^{-1}\bm{p}\bm{q}^{-})\bm{u}
\geq
\Delta^{-1}\bm{p}\bm{q}^{-}\bm{u}
=
\Delta^{-1}(\bm{q}^{-}\bm{u})\bm{p}
=
\alpha\bm{p},
$$
which yields the left inequality at \eqref{I-alphapxalphaDeltaq}.

Since $\bm{q}\bm{p}^{-}\geq(\bm{q}^{-}\bm{p})^{-1}\bm{I}=\Delta^{-1}\bm{I}$, we obtain the right inequality as follows
$$
\bm{x}
=
(\bm{I}\oplus\Delta^{-1}\bm{p}\bm{q}^{-})\bm{u}
\leq
(\bm{I}\oplus\bm{q}\bm{p}^{-}\bm{p}\bm{q}^{-})\bm{u}
=
(\bm{I}\oplus\bm{q}\bm{q}^{-})\bm{u}
=
(\bm{q}^{-}\bm{u})\bm{q}
=
\alpha\Delta\bm{q}.
$$

Now assume that the vector $\bm{x}$ satisfies \eqref{I-alphapxalphaDeltaq}, and then show that it can be written as \eqref{E-xIqppq}. From the right inequality at \eqref{I-alphapxalphaDeltaq}, it follows that
$$
\Delta^{-1}\bm{p}\bm{q}^{-}\bm{x}
\leq
\Delta^{-1}\bm{p}\bm{q}^{-}(\alpha\Delta\bm{q})
=
\alpha\bm{p}\bm{q}^{-}\bm{q}
=
\alpha\bm{p}.
$$

Considering the left inequality, we have $\bm{x}\geq\alpha\bm{p}\geq\Delta^{-1}\bm{p}\bm{q}^{-}\bm{x}$. Finally, after setting $\bm{u}=\bm{x}$, we can write
$$
\bm{x}
=
\bm{x}
\oplus
\Delta^{-1}\bm{p}\bm{q}^{-}\bm{x}
=
\bm{u}
\oplus
\Delta^{-1}\bm{p}\bm{q}^{-}\bm{u}
=
(\bm{I}\oplus\Delta^{-1}\bm{p}\bm{q}^{-})\bm{u},
$$
which gives a representation of the vector $\bm{x}$ in the form of \eqref{E-xIqppq}.

\section{Constrained Optimization Problems}
\label{S-COP}

We now add lower and upper boundary constraints on the feasible solutions. Given vectors $\bm{p},\bm{q},\bm{g},\bm{h}\in\mathbb{X}^{n}$, consider the problem to find all regular vectors $\bm{x}\in\mathbb{X}^{n}$ that
\begin{equation}
\begin{aligned}
&
\text{minimize}
&&
\bm{q}^{-}\bm{x}\bm{x}^{-}\bm{p},
\\
&
\text{subject to}
&&
\bm{g}
\leq
\bm{x}
\leq
\bm{h}.
\end{aligned}
\label{P-minqxxpgxh}
\end{equation}

The next theorem provides a complete direct solution to the problem.
\begin{theorem}
\label{T-minqxxpgxh}
Let $\bm{p}$ be a nonzero vector, $\bm{q}$ a regular vector, and $\bm{h}$ a regular vector such that $\bm{h}^{-}\bm{g}\leq\mathbb{1}$. Then, the minimum in problem \eqref{P-minqxxpgxh} is equal to
\begin{equation*}
\theta
=
\bm{q}^{-}(\bm{I}\oplus\bm{g}\bm{h}^{-})\bm{p},
\end{equation*}
and all regular solutions of the problem are given by
\begin{equation}
\bm{x}
=
(\bm{I}
\oplus
\theta^{-1}\bm{p}\bm{q}^{-})
\bm{u},
\qquad
\bm{g}
\leq
\bm{u}
\leq
(\bm{h}^{-}(\bm{I}\oplus\theta^{-1}\bm{p}\bm{q}^{-}))^{-}.
\label{E-xIthetapqu}
\end{equation}
\end{theorem}
\begin{proof}
As in the previous proof, we rewrite the objective function in the form $\bm{q}^{-}\bm{x}\bm{x}^{-}\bm{p}=\bm{x}^{-}\bm{A}\bm{x}$, where $\bm{A}=\bm{p}\bm{q}^{-}$, and thus reduce the problem to \eqref{P-minxAxgxh}.

Furthermore, we obtain the spectral radius of $\bm{A}$ to be equal to $\Delta=\bm{q}^{-}\bm{p}$, write $\bm{A}^{m}=\Delta^{m-1}\bm{p}\bm{q}^{-}$, and calculate $\bm{h}^{-}\bm{A}^{m}\bm{g}=\Delta^{m-1}\bm{h}^{-}\bm{p}\bm{q}^{-}\bm{g}$.

Then, we apply Theorem~\ref{T-minxAxgxh} to write the minimum value in the form
$$
\theta
=
\Delta\oplus\bigoplus_{m=1}^{n}(\Delta^{m-1}\bm{h}^{-}\bm{p}\bm{q}^{-}\bm{g})^{1/m}
=
\Delta\left(\mathbb{1}
\oplus
\bigoplus_{m=1}^{n}(\Delta^{-1}\bm{h}^{-}\bm{p}\bm{q}^{-}\bm{g})^{1/m}
\right).
$$

To simplify the last expression, consider two cases. First, we suppose that $\Delta\geq\bm{h}^{-}\bm{p}\bm{q}^{-}\bm{g}$. It follows immediately from this condition that the inequality $(\Delta^{-1}\bm{h}^{-}\bm{p}\bm{q}^{-}\bm{g})^{1/m}\leq\mathbb{1}$ holds for every $m$, and therefore, $\theta=\Delta$.

Otherwise, if the opposite inequality $\Delta<\bm{h}^{-}\bm{p}\bm{q}^{-}\bm{g}$ is satisfied, we see that $\Delta^{-1}\bm{h}^{-}\bm{p}\bm{q}^{-}\bm{g}\geq(\Delta^{-1}\bm{h}^{-}\bm{p}\bm{q}^{-}\bm{g})^{1/m}>\mathbb{1}$, which gives $\theta=\bm{h}^{-}\bm{p}\bm{q}^{-}\bm{g}$.

By combining both results and considering that $\Delta=\bm{q}^{-}\bm{p}$, we obtain the minimum value
$$
\theta
=
\Delta\oplus\bm{q}^{-}\bm{g}\bm{h}^{-}\bm{p}
=
\bm{q}^{-}(\bm{I}\oplus\bm{g}\bm{h}^{-})\bm{p}.
$$

We now define the solution set according to Theorem~\ref{T-minxAxgxh}. We first calculate $(\theta^{-1}\bm{A})^{m}=\theta^{-m}\Delta^{m-1}\bm{p}\bm{q}^{-}$, and then apply \eqref{E-Aast} to write
$$
(\theta^{-1}\bm{A})^{\ast}
=
\bm{I}
\oplus
\theta^{-1}\left(\bigoplus_{m=1}^{n-1}(\theta^{-1}\Delta)^{m-1}\right)\bm{p}\bm{q}^{-}.
$$

Since $\theta\geq\Delta$, the inequality $(\theta^{-1}\Delta)^{m-1}\leq\mathbb{1}$ is valid for all $m$ and becomes an equality if $m=1$. Therefore, the term in the parenthesis on the right-hand side is equal to $\mathbb{1}$, and hence
$$
(\theta^{-1}\bm{A})^{\ast}
=
\bm{I}\oplus\theta^{-1}\bm{p}\bm{q}^{-}.
$$

Substitution into the solution provided by Theorem~\ref{T-minxAxgxh} leads to \eqref{E-xIthetapqu}.
\end{proof}

Suppose that we replace the simple boundary constraints in the above problem by a linear inequality constraint given by a matrix $\bm{B}\in\mathbb{X}^{n\times n}$ and a vector $\bm{g}\in\mathbb{X}^{n}$. Consider the problem to find regular vectors $\bm{x}\in\mathbb{X}^{n}$ that
\begin{equation}
\begin{aligned}
&
\text{minimize}
&&
\bm{q}^{-}\bm{x}\bm{x}^{-}\bm{p},
\\
&
\text{subject to}
&&
\bm{B}\bm{x}\oplus\bm{g}
\leq
\bm{x}.
\end{aligned}
\label{P-minqxxpBxgx}
\end{equation}

A solution to the problem can be obtained as follows.
\begin{theorem}
\label{T-minqxxpBxgx}
Let $\bm{p}$ be a nonzero vector, $\bm{q}$ a regular vector, and $\bm{B}$ be a matrix such that $\mathop\mathrm{Tr}(\bm{B})\leq\mathbb{1}$. Then, the minimum value in problem \eqref{P-minqxxpBxgx} is equal to
\begin{equation*}
\theta
=
\bm{q}^{-}\bm{B}^{\ast}\bm{p},
\end{equation*}
and all regular solutions of the problem are given by
\begin{equation}
\bm{x}
=
(\theta^{-1}\bm{p}\bm{q}^{-}\oplus\bm{B})^{\ast}\bm{u},
\qquad
\bm{u}
\geq
\bm{g}.
\label{E-xItheta1pqBu-ug}
\end{equation}
\end{theorem}
\begin{proof}
To solve the problem, we again represent the objective function as $\bm{x}^{-}\bm{p}\bm{q}^{-}\bm{x}$ and then apply Theorem~\ref{T-minxAxBxgx} with $\bm{A}=\bm{p}\bm{q}^{-}$. First, we examine the minimum value provided by Theorem~\ref{T-minxAxBxgx}. This minimum now takes the form
$$
\theta
=
\bigoplus_{k=1}^{n}\mathop{\bigoplus\hspace{1.1em}}_{0\leq i_{1}+\cdots+i_{k}\leq n-k}\mathop\mathrm{tr}\nolimits^{1/k}(\bm{p}\bm{q}^{-}\bm{B}^{i_{1}}\cdots\bm{p}\bm{q}^{-}\bm{B}^{i_{k}}),
$$
where the properties of the trace allow us to write
$$
\mathop\mathrm{tr}(\bm{p}\bm{q}^{-}\bm{B}^{i_{1}}\cdots\bm{p}\bm{q}^{-}\bm{B}^{i_{k}})
=
\bm{q}^{-}\bm{B}^{i_{1}}\bm{p}\cdots\bm{q}^{-}\bm{B}^{i_{k}}\bm{p}.
$$

By truncating the sum at $k=1$, we bound the value of $\theta$ from below as
$$
\theta
\geq
\bigoplus_{i=0}^{n-1}\mathop\mathrm{tr}(\bm{p}\bm{q}^{-}\bm{B}^{i})
=
\bigoplus_{i=0}^{n-1}\bm{q}^{-}\bm{B}^{i}\bm{p}
=
\bm{q}^{-}\bm{B}^{\ast}\bm{p}.
$$

On the other hand, it follows from the Carr\`{e} inequality that $\bm{B}^{m}\leq\bm{B}^{\ast}$ for any integer $m\geq0$, and hence we can write
$$
\bm{q}^{-}\bm{B}^{i_{1}}\bm{p}\cdots\bm{q}^{-}\bm{B}^{i_{k}}\bm{p}
\leq
\bm{q}^{-}\bm{B}^{\ast}\bm{p}\cdots\bm{q}^{-}\bm{B}^{\ast}\bm{p}
=
(\bm{q}^{-}\bm{B}^{\ast}\bm{p})^{k}.
$$

Consequently, we have the inequality
$$
\theta
=
\bigoplus_{k=1}^{n}\mathop{\bigoplus\hspace{1.1em}}_{0\leq i_{1}+\cdots+i_{k}\leq n-k}(\bm{q}^{-}\bm{B}^{i_{1}}\bm{p}\cdots\bm{q}^{-}\bm{B}^{i_{k}}\bm{p})^{1/k}
\leq
\bm{q}^{-}\bm{B}^{\ast}\bm{p},
$$
which together with the opposite inequality yields the desired result.

Finally, we use Theorem~\ref{T-minxAxBxgx} to write the solution in the form of \eqref{E-xItheta1pqBu-ug}, and thus complete the proof.
\end{proof}

\section{Application to Project Scheduling}
\label{AJS}

In this section, we show how the results obtained can be applied to solve real-world problems that are drawn from project scheduling under the minimum makespan criterion (see, e.g., \cite{Demeulemeester2002Project,Tkindt2006Multicriteria,Vanhoucke2013Project} for further details).

We start with the problem, which involves only the start-finish temporal constraints. Consider the standard representation of the problem in the form of \eqref{P-minmaxxcmaxx}, and rewrite it in the tropical mathematics setting.

In terms of the operations in the semifield $\mathbb{R}_{\max,+}$, the problem becomes
\begin{equation*}
\begin{aligned}
&
\text{minimize}
&&
\bigoplus_{i=1}^{n}\bigoplus_{j=1}^{n}c_{ij}x_{j}
\bigoplus_{k=1}^{n}x_{k}^{-1}.
\end{aligned}
\end{equation*}

Furthermore, we introduce the matrix $\bm{C}=(c_{ij})$ and the vector $\bm{x}=(x_{i})$. Using this notation and considering that $\bm{1}=(0,\ldots,0)^{T}$ for $\mathbb{R}_{\max,+}$, we put the problem in the vector form
\begin{equation}
\begin{aligned}
&
\text{minimize}
&&
\bm{1}^{T}\bm{C}\bm{x}\bm{x}^{-}\bm{1}.
\end{aligned}
\label{P-min1TCxx1}
\end{equation}

The last problem is a special case of problem \eqref{P-minqxxp}, where we take $\bm{p}=\bm{1}$ and $\bm{q}^{-}=\bm{1}^{T}\bm{C}$. Note that the vector $\bm{q}$ is regular if the matrix $\bm{C}$ is column-regular.

As a consequence of Theorems~\ref{T-minqxxp} and \ref{T-minxpqx}, we obtain the following result.
\begin{theorem}
\label{T-min1TCxx1}
Let $\bm{C}$ be a column-regular matrix. Then, the minimum value in problem \eqref{P-min1TCxx1} is equal to
\begin{equation*}
\Delta
=
\bm{1}^{T}\bm{C}\bm{1}
=
\|\bm{C}\|,
\end{equation*}
and all regular solutions of the problem are given by
\begin{equation}
\alpha\bm{1}
\leq
\bm{x}
\leq
\alpha\Delta(\bm{1}^{T}\bm{C})^{-},
\qquad
\alpha\in\mathbb{R},
\label{E-alpa1xalphaDelta1C}
\end{equation}
or, equivalently, by
\begin{equation}
\bm{x}
=
(\bm{I}
\oplus
\Delta^{-1}\bm{1}\bm{1}^{T}\bm{C})
\bm{u},
\qquad
\bm{u}\in\mathbb{R}^{n}.
\label{E-xIDelta11Cu}
\end{equation}
\end{theorem}

We now examine the problem with the early start and late finish constraints, represented as \eqref{P-minmaxxcmaxx-xcf-gx}. We take the vectors $\bm{g}=(g_{i})$ and $\bm{f}=(f_{i})$, and rewrite the problem in terms of the semifield $\mathbb{R}_{\max,+}$. As a result, we extend the unconstrained problem at \eqref{P-min1TCxx1} to the problem
\begin{equation}
\begin{aligned}
&
\text{minimize}
&&
\bm{1}^{T}\bm{C}\bm{x}\bm{x}^{-}\bm{1},
\\
&
\text{subject to}
&&
\bm{C}\bm{x}
\leq
\bm{f},
\\
&&&
\bm{g}
\leq
\bm{x}.
\end{aligned}
\label{P-min1TCxx1gxCxf}
\end{equation}

It follows from Lemma~\ref{L-Axd} that the first inequality constraint can be solved in the form $\bm{x}\leq(\bm{f}^{-}\bm{C})^{-}$. Then, the problem reduces to \eqref{P-minqxxpgxh} with $\bm{p}=\bm{1}$, $\bm{q}^{-}=\bm{1}^{T}\bm{C}$ and $\bm{h}=(\bm{f}^{-}\bm{C})^{-}$. By applying Theorem~\ref{T-minqxxpgxh} and using properties of the norm to represent the minimum value, we come to the following result.
\begin{theorem}
\label{T-min1TCxx1gxCxf}
Let $\bm{C}$ be a column-regular matrix, and $\bm{f}$ a regular vector such that $\bm{f}^{-}\bm{C}\bm{g}\leq\mathbb{1}$. Then, the minimum value in problem \eqref{P-min1TCxx1gxCxf} is equal to
\begin{equation*}
\theta
=
\bm{1}^{T}\bm{C}(\bm{I}\oplus\bm{g}\bm{f}^{-}\bm{C})\bm{1}
=
\|\bm{C}\|\oplus\|\bm{C}\bm{g}\|\|\bm{f}^{-}\bm{C})\|,
\end{equation*}
and all regular solutions of the problem are given by
\begin{equation}
\bm{x}
=
(\bm{I}
\oplus
\theta^{-1}\bm{1}\bm{1}^{T}\bm{C})
\bm{u},
\qquad
\bm{g}
\leq
\bm{u}
\leq
(\bm{f}^{-}\bm{C}(\bm{I}\oplus\theta^{-1}\bm{1}\bm{1}^{T}\bm{C}))^{-}.
\label{E-xItheta111Cu}
\end{equation}
\end{theorem}

Finally, we consider problem \eqref{P-minmaxxcmaxx-xcdx-gx} to represent it in terms of $\mathbb{R}_{\max,+}$. We define the matrix $\bm{D}=(d_{ij})$, and then write
\begin{equation}
\begin{aligned}
&
\text{minimize}
&&
\bm{1}^{T}\bm{C}\bm{x}\bm{x}^{-}\bm{1},
\\
&
\text{subject to}
&&
\bm{D}\bm{C}\bm{x}\oplus\bm{g}
\leq
\bm{x}.
\end{aligned}
\label{P-min1TCxx1DCxgx}
\end{equation}

Application of Theorem~\ref{T-minqxxpBxgx} with $\bm{p}=\bm{1}$, $\bm{q}^{-}=\bm{1}^{T}\bm{C}$ and $\bm{B}=\bm{D}\bm{C}$ yields the next result.

\begin{theorem}
\label{T-min1TCxx1DCxgx}
Let $\bm{C}$ be a column-regular matrix and $\bm{D}$ be a matrix such that $\mathop\mathrm{Tr}(\bm{D}\bm{C})\leq\mathbb{1}$. Then, the minimum value in problem \eqref{P-min1TCxx1DCxgx} is equal to
\begin{equation*}
\theta
=
\bm{1}^{T}\bm{C}(\bm{D}\bm{C})^{\ast}\bm{1}
=
\|\bm{C}(\bm{D}\bm{C})^{\ast}\|,
\end{equation*}
and all regular solutions of the problem are given by
\begin{equation}
\bm{x}
=
(\theta^{-1}\bm{1}\bm{1}^{T}\bm{C}\oplus\bm{D}\bm{C})^{\ast}\bm{u},
\qquad
\bm{u}
\geq
\bm{g}.
\label{E-xtheta111CDCu}
\end{equation}
\end{theorem}

Note that the solutions obtained are not unique. In the context of project scheduling, this leaves the freedom to account for additional temporal constraints.

\section{Numerical Examples}
\label{S-NE}

The main aim of this section is to provide a transparent and detailed illustration of the computational technique in terms of the semifield $\mathbb{R}_{\max,+}$, which is used to obtain the solution. To this end, we examine relatively artificial low-dimensional problems which, however, clearly demonstrate the ability of the approach to solve real-world problems of high dimensions. 

Consider an example project that involves $n=3$ activities and operates under the constraints given by
$$
\bm{C}
=
\left(
\begin{array}{ccr}
4 & 0 & \mathbb{0}
\\
1 & 3 & -1
\\
0 & 2 & 2
\end{array}
\right),
\qquad
\bm{D}
=
\left(
\begin{array}{ccc}
\mathbb{0} & \mathbb{0} & \mathbb{0}
\\
0 & \mathbb{0} & \mathbb{0}
\\
2 & 1 & \mathbb{0}
\end{array}
\right),
\qquad
\bm{g}
=
\left(
\begin{array}{c}
3
\\
2
\\
1
\end{array}
\right),
\qquad
\bm{f}
=
\left(
\begin{array}{c}
8
\\
7
\\
4
\end{array}
\right),
$$
where the symbol $\mathbb{0}=-\infty$ is employed to save writing.

We start with problem \eqref{P-min1TCxx1}, where only the start-finish constraints are defined. Application of Theorem~\ref{T-min1TCxx1} gives the minimum makespan
$$
\Delta
=
\|\bm{C}\|
=
4.
$$

To represent the solution to the problem in the form of \eqref{E-alpa1xalphaDelta1C}, we take the vector $\bm{1}=(0,0,0)^{T}$ and calculate
$$
\bm{1}^{T}\bm{C}
=
\left(
\begin{array}{ccc}
4 & 3 & 2
\end{array}
\right),
\qquad
(\bm{1}^{T}\bm{C})^{-}
=
\left(
\begin{array}{c}
-4
\\
-3
\\
-2
\end{array}
\right),
\qquad
\Delta(\bm{1}^{T}\bm{C})^{-}
=
\left(
\begin{array}{c}
0
\\
1
\\
2
\end{array}
\right).
$$

According to the theorem, the solution vector $\bm{x}=(x_{1},x_{2},x_{3})^{T}$, which describes the initiation time of activities, satisfies the following inequality
$$
\alpha
\left(
\begin{array}{c}
0
\\
0
\\
0
\end{array}
\right)
\leq
\bm{x}
\leq
\alpha
\left(
\begin{array}{c}
0
\\
1
\\
2
\end{array}
\right),
\qquad
\alpha\in\mathbb{R}.
$$

In terms of the usual operations, the solution takes the form
$$
x_{1}
=
\alpha
\qquad
\alpha
\leq
x_{2}
\leq
\alpha+1,
\qquad
\alpha
\leq
x_{3}
\leq
\alpha+2,
\qquad
\alpha\in\mathbb{R}.
$$

Furthermore, we turn to the representation defined by \eqref{E-xIDelta11Cu}. After calculating the matrices 
$$
\bm{1}\bm{1}^{T}\bm{C}
=
\left(
\begin{array}{ccc}
4 & 3 & 2
\\
4 & 3 & 2
\\
4 & 3 & 2
\end{array}
\right),
\qquad
\bm{I}\oplus\Delta^{-1}\bm{1}\bm{1}^{T}\bm{C}
=
\left(
\begin{array}{rrr}
0 & -1 & -2
\\
0 & 0 & -2
\\
0 & -1 & 0
\end{array}
\right),
$$
we obtain the solution
$$
\bm{x}
=
\left(
\begin{array}{rrr}
0 & -1 & -2
\\
0 & 0 & -2
\\
0 & -1 & 0
\end{array}
\right)
\bm{u},
\qquad
\bm{u}\in\mathbb{R}^{3}.
$$

In the ordinary notation, assuming $\bm{u}=(u_{1},u_{2},u_{3})^{T}$, we have
\begin{align*}
x_{1}
&=
\max(u_{1},u_{2}-1,u_{3}-2),
\\
x_{2}
&=
\max(u_{1},u_{2},u_{3}-2),
\\
x_{3}
&=
\max(u_{1},u_{2}-1,u_{3}).
\end{align*}

Suppose that, in addition to the start-finish constraints, both early start and late finish constraints are also imposed. To check whether Theorem~\ref{T-min1TCxx1gxCxf} can be applied, we first obtain
$$
\bm{f}^{-}\bm{C}
=
\left(
\begin{array}{rrr}
-4 & -2 & -2
\end{array}
\right),
\qquad
\bm{C}\bm{g}
=
\left(
\begin{array}{c}
7
\\
5
\\
4
\end{array}
\right),
\qquad
\bm{f}^{-}\bm{C}\bm{g}
=
0.
$$

Since $\bm{f}^{-}\bm{C}\bm{g}=0=\mathbb{1}$, we see that the conditions of Theorem~\ref{T-min1TCxx1gxCxf} are fulfilled. Considering that
$$
\|\bm{C}\|
=
4,
\qquad
\|\bm{f}^{-}\bm{C}\|
=
-2,
\qquad
\|\bm{C}\bm{g}\|
=
7,
$$
we evaluate the minimum makespan
$$
\theta
=
\|\bm{C}\|\oplus\|\bm{C}\bm{g}\|\|\bm{f}^{-}\bm{C}\|
=
5.
$$

Furthermore, we successively calculate the matrices
$$
\theta^{-1}\bm{1}\bm{1}^{T}\bm{C}
=
\left(
\begin{array}{rrr}
-1 & -2 & -3
\\
-1 & -2 & -3
\\
-1 & -2 & -3
\end{array}
\right),
\qquad
\bm{I}\oplus\theta^{-1}\bm{1}\bm{1}^{T}\bm{C}
=
\left(
\begin{array}{rrr}
0 & -2 & -3
\\
-1 & 0 & -3
\\
-1 & -2 & 0
\end{array}
\right),
$$
and the vector
$$
\bm{f}^{-}\bm{C}(\bm{I}\oplus\theta^{-1}\bm{1}\bm{1}^{T}\bm{C})
=
\left(
\begin{array}{ccc}
-3 & -2 & -2
\end{array}
\right).
$$

The solution given by Theorem~\ref{T-min1TCxx1gxCxf} at \eqref{E-xItheta111Cu} takes the form
$$
\bm{x}
=
\left(
\begin{array}{rrr}
0 & -2 & -3
\\
-1 & 0 & -3
\\
-1 & -2 & 0
\end{array}
\right)
\bm{u},
\qquad
\left(
\begin{array}{c}
3
\\
2
\\
1
\end{array}
\right)
\leq
\bm{u}
\leq
\left(
\begin{array}{c}
3
\\
2
\\
2
\end{array}
\right).
$$

Scalar representation in the ordinary notation gives the equalities
\begin{align*}
x_{1}
&=
\max(u_{1},u_{2}-2,u_{3}-3),
\\
x_{2}
&=
\max(u_{1}-1,u_{2},u_{3}-3),
\\
x_{3}
&=
\max(u_{1}-1,u_{2}-2,u_{3}),
\end{align*}
where the numbers $u_{1}$, $u_{2}$ and $u_{3}$ satisfy the conditions
$$
u_{1}
=
3,
\qquad
u_{2}
=
2,
\qquad
1
\leq
u_{3}
\leq
2.
$$

By combining these conditions with the equalities, we obtain the single solution, which determines the optimal initiation time of activities as
$$
x_{1}
=
3,
\qquad
x_{2}
=
2,
\qquad
x_{3}
=
2.
$$

Finally, we assume that the start-finish constraints given by the matrix $\bm{D}$ are defined instead of the late finish constraints. To solve the problem, we use the result provided by Theorem~\ref{T-min1TCxx1DCxgx}. First, we calculate the matrices
$$
\bm{D}\bm{C}
=
\left(
\begin{array}{ccc}
\mathbb{0} & \mathbb{0} & \mathbb{0}
\\
4 & 0 & \mathbb{0}
\\
6 & 4 & 0
\end{array}
\right),
\qquad
(\bm{D}\bm{C})^{2}
=
\left(
\begin{array}{ccc}
\mathbb{0} & \mathbb{0} & \mathbb{0}
\\
4 & 0 & \mathbb{0}
\\
8 & 4 & 0
\end{array}
\right),
\qquad
(\bm{D}\bm{C})^{3}
=
\left(
\begin{array}{ccc}
\mathbb{0} & \mathbb{0} & \mathbb{0}
\\
4 & 0 & \mathbb{0}
\\
8 & 4 & 0
\end{array}
\right).
$$

Application of \eqref{E-TrA} yields
$$
\mathop\mathrm{Tr}(\bm{D}\bm{C})=0=\mathbb{1},
$$
which means that the conditions of Theorem~\ref{T-min1TCxx1DCxgx} are satisfied.

Furthermore, we have
$$
(\bm{D}\bm{C})^{\ast}
=
\left(
\begin{array}{ccc}
0 & \mathbb{0} & \mathbb{0}
\\
4 & 0 & \mathbb{0}
\\
8 & 4 & 0
\end{array}
\right),
\qquad
\bm{C}(\bm{D}\bm{C})^{\ast}
=
\left(
\begin{array}{ccr}
4 & 0 & \mathbb{0}
\\
7 & 3 & -1
\\
10 & 6 & 2
\end{array}
\right),
$$
and then find the minimum makespan
$$
\theta
=
\|\bm{C}(\bm{D}\bm{C})^{\ast}\|
=
10.
$$

It remains to represent the solution. We calculate the matrices
\begin{gather*}
\theta^{-1}\bm{1}\bm{1}^{T}\bm{C}
=
\left(
\begin{array}{rrr}
-6 & -7 & -8
\\
-6 & -7 & -8
\\
-6 & -7 & -8
\end{array}
\right),
\qquad
\theta^{-1}\bm{1}\bm{1}^{T}\bm{C}\oplus\bm{D}\bm{C}
=
\left(
\begin{array}{rrr}
-6 & -7 & -8
\\
0 & -7 & -8
\\
2 & 1 & -8
\end{array}
\right),
\\
(\theta^{-1}\bm{1}\bm{1}^{T}\bm{C}\oplus\bm{D}\bm{C})^{2}
=
\left(
\begin{array}{rrr}
-6 & -7 & -14
\\
-6 & -7 & -8
\\
1 & -5 & -6
\end{array}
\right),
\qquad
(\theta^{-1}\bm{1}\bm{1}^{T}\bm{C}\oplus\bm{D}\bm{C})^{\ast}
=
\left(
\begin{array}{crr}
0 & -7 & -8
\\
0 & 0 & -8
\\
2 & 1 & 0
\end{array}
\right).
\end{gather*}

The solution defined by \eqref{E-xtheta111CDCu} becomes
$$
\bm{x}
=
\left(
\begin{array}{crr}
0 & -7 & -8
\\
0 & 0 & -8
\\
2 & 1 & 0
\end{array}
\right)
\bm{u},
\qquad
\bm{u}
\geq
\left(
\begin{array}{c}
3
\\
2
\\
1
\end{array}
\right).
$$

By rewriting the solution in the usual notation, we have
\begin{align*}
x_{1}
&=
\max(u_{1},u_{2}-7,u_{3}-8),
\\
x_{2}
&=
\max(u_{1},u_{2},u_{3}-8),
\\
x_{3}
&=
\max(u_{1}+2,u_{2}+1,u_{3}),
\end{align*}
under the conditions that
$$
u_{1}
\geq
3,
\qquad
u_{2}
\geq
2,
\qquad
u_{3}
\geq
1.
$$

Specifically, with $u_{1}=3$, $u_{2}=2$ and $u_{3}=1$, we obtain the earliest optimal initiation times given by $x_{1}=3$, $x_{2}=3$ and $x_{3}=5$.

\bibliographystyle{abbrvurl}
\bibliography{Tropical_optimization_problems_with_application_to_project_scheduling_with_minimum_makespan}

\end{document}